\documentclass[12pt]{article}
\usepackage{mathrsfs}
\usepackage{t1enc}
\usepackage{amsmath}
\usepackage[latin1]{inputenc}
\usepackage[english]{babel}
\usepackage{amsmath,amsthm}
\usepackage{amsfonts}
\usepackage{latexsym,lineno}
\usepackage{listings}
\usepackage{hyperref}

\usepackage[natural]{xcolor}

\usepackage{graphicx, epsfig, subfigure}
\usepackage{epstopdf}
\usepackage{enumerate} 

\usepackage{geometry}
\usepackage{caption}

\theoremstyle{plain}

\newtheorem{problem}{Problem}

\newtheorem{theorem}{Theorem}
\newtheorem{lemma}{Lemma}
\newtheorem{corollary}{Corollary}

\numberwithin{conjecture}{section}
\newtheorem*{conjecture*}{Conjecture}

\theoremstyle{definition}

\theoremstyle{remark}

\theoremstyle{plain}

\theoremstyle{plain}

\theoremstyle{plain}

\usepackage{lineno}

\title{On splitting digraphs}
\author{Donglei Yang\unskip\textsuperscript{a,}\thanks{\emph{E-mail address:} dlyang120@163.com.}, Yandong Bai\unskip\textsuperscript{b,}\thanks{\emph{E-mail address:} bai@nwpu.edu.cn.}, Guanghui Wang\unskip\textsuperscript{a,}\thanks{Corresponding author. \emph{E-mail address}: ghwang@sdu.edu.cn}, Jianliang Wu\unskip\textsuperscript{a,}\thanks{\emph{E-mail address:} jlwu@sdu.edu.cn.}
\\
{\small \textsuperscript{a }\unskip School of Mathematics,}\\
{\small Shandong University, Jinan, 250100, P. R. China}\\
{\small \textsuperscript{b }\unskip  Department of Applied Mathematics,}\\
{\small Northwestern Polytechnical University, Xi'an, 710072, P. R. China}\\
}

\date{}

\begin{document}
\baselineskip 0.65cm

\maketitle
\begin{abstract}
In 1995, Stiebitz asked the following question: For any positive integers $s,t$, is there a finite integer $f(s,t)$ such that every digraph $D$ with minimum out-degree at least $f(s,t)$ admits a bipartition $(A, B)$ such that $A$ induces a subdigraph with minimum out-degree at least $s$ and $B$ induces a subdigraph with minimum out-degree at least $t$? We give an affirmative answer for tournaments, multipartite tournaments, and digraphs with bounded maximum in-degrees. In particular, we show that for every $\epsilon$ with $0<\epsilon<1/2$, there exists an integer $\delta_0$ such that every tournament with minimum out-degree at least $\delta_0$ admits a bisection $(A, B)$, so that each vertex has at least $(1/2-\epsilon)$ of its out-neighbors in $A$, and in $B$ as well.
\end{abstract}
\section{Introduction}


Partitioning an undirected graph (a digraph) into two parts under certain constraints (e.g., see \cite{kuhn} for connectivity constraint, see \cite{stie2} for chromatic number constraint)
has been widely studied due to its important applications in induction arguments.
Among them, partitions under degree constraints have attracted special attention and a number of classical results in undirected graphs have been achieved.
Lov\'{a}sz \cite{lovasz} proved in 1966 that every undirected graph with maximum degree $s+t+1$
can be partitioned into two parts such that they induce two subgraphs with maximum degree at most $s$ and at most $t$, respectively.
Stiebitz \cite{stie} showed in 1996 that every undirected graph with minimum degree $s+t+1$
can be partitioned into two parts such that they induce two subgraphs with minimum degree at least $s$ and at least $t$, respectively.
A natural question is whether or not the corresponding assertions for digraphs hold, where the degree is replaced by out-degree.

For Lov\'{a}sz's result under maximum degree constraint,
Alon \cite{alon} pointed out that its corresponding assertion fails for digraphs in the following strong sense:
For every $k$, there is a digraph without even cycles, in which all out-degrees are exactly $k$. It is trivial to prove that for every bipartition of such a digraph, the maximum out-degree in one of the two parts is $k$. This example was given by Thomassen \cite{even cycle} in 1985.
How about the corresponding assertion for digraphs with respect to Stiebitz's result?
In fact, in 1995, Stiebitz \cite{stie3} proposed the following problem.

\begin{problem}\label{p1}
  For any positive integers $s,t$, is there a finite integer $f(s,t)$ such that every digraph with minimum out-degree at least $f(s,t)$ admits a bipartition $(A, B$), so that $A$ induces a subdigraph with minimum out-degree at least $s$ and $B$ induces a subdigraph with minimum out-degree at least $t$?
\end{problem}
 This problem was also mentioned in \cite{alon}. For general digraphs, the only known value is $f(1,1)=3$ from a result of Thomassen \cite{thom}. Lichiardopol \cite{lichi} proved that every tournament with minimum out-degree at least $t+\frac{s^2+3s+2}{2}$ admits a bipartition $(A, B)$ such that $A$ induces a subdigraph with minimum out-degree at least $s$ and $B$ induces a subdigraph with minimum out-degree at least $t$ .
K\'{e}zdy \cite{andre} constructed an example showing that $f(2,2)>5$. For more results about splitting digraphs, the readers are referred to \cite{hou,lee1}.


Particularly, we ask for a bipartition $(A,B)$ with $A$ and $B$ of fixed sizes. A \emph{bisection} is a bipartition $(A,B)$ with $||A|-|B||\leq1$.
Bollob\'{a}s and Scott \cite{bollo} conjectured that every graph $G$ has a bisection $(A,B)$ with $d_H(v)\geq\frac{d_G(v)-1}{2}$ for each $v\in V(G)$, where $H$ is the subgraph induced by the set of edges between $A$ and $B$. However, Ji et al. \cite{majie} gave an infinite family of counterexamples to this conjecture, which indicates
that $\lfloor\frac{d_G(v)-1}{2}\rfloor$ is probably the correct lower bound. This conjecture is widely open and readers are referred to \cite{bollo,majie,lee2}.

For digraphs, unfortunately, the same example given by Thomassen \cite{even cycle} indicates that we cannot obtain a bipartition $(A, B)$ such that each vertex in one part has at least one out-neighbor in the other part.
So we begin to consider whether or not there exists a bisection of any digraph such that the two subdigraphs induced by the two parts have high minimum out-degree, and we propose the following problem.

\begin{problem}\label{p2}
    For any positive integers $s,t$, is there a finite integer $f(s,t)$ such that every digraph with minimum out-degree at least $f(s,t)$ admits a bisection $(A, B$), so that $A$ induces a subdigraph with minimum out-degree at least $s$ and $B$ induces a subdigraph with minimum out-degree at least $t$?
\end{problem}

In this paper, we give affirmative answers to Problems \ref{p1} and \ref{p2} for some classes of digraphs.
Given a digraph $D$ and a bipartition $(A,B)$ of $V(D)$, let $\Delta^-(D)$ be the maximum in-degree of $D$ and $e(A,B)$ be the number of arcs from $A$ to $B$. Let $D[A]$ and $D[B]$ be the induced subdigraphs of $D$ on $A$ and $B$, respectively. For a vertex $v\in V(D)$, we write $d^+_A(v)$ for the number of out-neighbors of $v$ in $A$, and $d^+_{B}(v)$ for the number of out-neighbors of $v$ in $B$. All digraphs considered here are simple (without loops or multiple arcs).

An $n$-partite tournament with $n\geq 2$, or multipartite tournament, is an orientation of a complete $n$-partite graph, and particularly, a tournament is an orientation of a complete graph. A digraph is \emph{strong} if, for every two vertices $x$ and $y$, there exists an $(x,y)$-path.
As for tournaments, we have the desired result in the following strong sense.
\begin{theorem}\label{th1}
  For every $\epsilon$ with $0<\epsilon<\frac{1}{2}$, there exists an integer $\delta_0$ such that every tournament $T$ with $\delta^+(T)\geq\delta_0$ admits a bisection $(A, B)$ with $\min\{d_A^+(v), d_{B}^+(v)\}\geq (\frac{1}{2}-\epsilon) d_T^+(v)$ for every $v\in V(T)$.
\end{theorem}
This result gives affirmative answers to Problems \ref{p1} and \ref{p2} for tournaments.
 A digraph with minimum out-degree $s$ is $s$-\emph{minimal} if any proper subdigraph has minimum out-degree at most $s-1$. It is not hard to prove that any $s$-minimal tournament $T$ with $s>0$ is strong. So we have the following result.

\begin{corollary}
  For every $\epsilon$ with $0<\epsilon<\frac{1}{2}$, there exists an integer $\delta_0$ such that every tournament $T$ with $\delta^+(T)\geq \delta_0$ admits a bipartition $(A, B)$ such that $T[A]$ is strong and $d_A^+(v)\geq (\frac{1}{2}-\epsilon) d_T^+(v)$ for every $v\in A$, $d_B^+(v)\geq (\frac{1}{2}-\epsilon) d_T^+(v)$ for every $v\in B$.
\end{corollary}
In fact, we can start with a bisection $(A,B)$ from Theorem \ref{th1}, such that $\min\{d_A^+(v), d_{B}^+(v)\}\geq (\frac{1}{2}-\epsilon) d_T^+(v)$ for every $v\in V(T)$. By moving vertices from $A$ to $B$, we have a minimal subset $A'\subset A$ such that $d_{A'}^+(v)\geq (\frac{1}{2}-\epsilon) d_T^+(v)$ for every $v\in A'$. Clearly, $T[A']$ is strong and $d_{B'}^+(v)\geq (\frac{1}{2}-\epsilon) d_T^+(v)$ for every $v\in B'=V(T)\setminus A'$.

By the weighted Lov\'{a}sz Local Lemma \cite{Michael}, we have the following theorem.
\begin{theorem}\label{thm2}
  For every $0<\epsilon<\frac{1}{2}$, there exists an integer $\delta_0$ such that every digraph $D$ with $\delta^+(D)\geq \delta_0$ and $\Delta^-(D)\leq \frac{e^{\epsilon^2(\delta^+(D)-1)}}{8\delta^+(D)}$ admits a bisection $(A, B)$ with $\min\{d_A^+(v), d_{B}^+(v)\}\geq (\frac{1}{2}-\epsilon) d_D^+(v)$ for every $v\in V(D)$.
\end{theorem}

For bipartite tournaments, we also have an affirmative answer to Problem \ref{p1}.
\begin{theorem}\label{thm3}
For any positive integers $s\leq t$, if $D$ is a bipartite tournament with $\delta^+(D)\geq t+\frac{(s+1)^4}{4s}-s$, then $D$ has a bipartition $(A, B)$ with $\delta^+(D[A])\geq s$ and $\delta^+(D[B])\geq t$.
\end{theorem}
For $k$-partite tournaments with $k>2$, we derive the following result from similar arguments in the proof of Theorem \ref{thm3}, and here we only give a trivial bound.
\begin{corollary}\label{cor3}
  For any positive integers $s\leq t$ and $k>2$, if $D$ is a $k$-partite tournament with $\delta^+(D)\geq t+\max\{2s(s+1)^2, 2ks(s+1)\}$, then $D$ has a bipartition $(A, B)$ with $\delta^+(D[A])\geq s$ and $\delta^+(D[B])\geq t$.
\end{corollary}

\section{Proofs of Theorem \ref{th1}, \ref{thm2} and \ref{thm3}}
\begin{proof}[{\bf Proof of Theorem \ref{th1}}]
Let $T$ be such a tournament with $n$ vertices, and assume $n$ is even. We arbitrarily partition the vertices of $T$ into disjoint pairs $\{v_1,w_1\},\{v_2,w_2\},\ldots,\{v_{n/2},w_{n/2}\}$ (we allow a singleton when $n$ is odd and deal with it in the similar method) and separate each pair independently and uniformly, then we have a bisection $(A,  B)$.

For a vertex $v\in A$ (or $B$) in the pair $\{v,w\}$, let $X_v$ be the number of out-neighbors of vertex $v$ in $A$ (or $B$). We say $v$ is \emph{bad} if either $X_v<t:=\lceil(\frac{1}{2}-\epsilon)d_T^+(v)\rceil$ or $X_v>d_T^+(v)-t$ and denote by $X$ the number of bad vertices.
For every $v\in V(T)$, let $a_v=|\{i\in [\frac{n}{2}] : \{v_i,w_i\}\subseteq N^+(v)\}|$ and $b_v=|\{i\in [\frac{n}{2}] : |N^+(v)\cap \{v_i,w_i\}|=1\}|$. Thus we have $d^+_T(v)=2a_v+b_v$ and $Pr(X_v<t)=Pr(X_v>d^+_T(v)-t)=0$ when $a_v\geq t$. Consider $a_v<t$, we have
$$Pr(X_v<t)=\left\{\begin{array}{cc}
                                            \sum\limits^{t-1-a_v}_{i=0}
                                            \binom {b_v-1} i
                                            \left(\frac{1}{2}\right)^{b_v-1} & w\in N^+(v), \\
                                            \sum\limits^{t-1-a_v}_{i=0}
                                            \binom {b_v} i
                                            \left(\frac{1}{2}\right)^{b_v}
                                             & w\in N^-(v).
                                          \end{array}\right.
$$
Similarly, we have
$$Pr(X_v>d^+_T(v)-t)=\left\{\begin{array}{cc}
                                            \sum\limits^{t-2-a_v}_{i=0}
                                            \binom {b_v-1} i
                                            \left(\frac{1}{2}\right)^{b_v-1} & w\in N^+(v), \\
                                            \sum\limits^{t-1-a_v}_{i=0}
                                            \binom {b_v} i
                                            \left(\frac{1}{2}\right)^{b_v}
                                             & w\in N^-(v).
                                          \end{array}\right.$$
If $a_v<t$, then $\binom {b_v} i\left(\frac{1}{2}\right)^{b_v}\leq\binom {b_v-1} i\left(\frac{1}{2}\right)^{b_v-1}$ for each $i$ with $0\leq i\leq t-1-a_v$, and it follows that$\sum\limits^{t-1-a_v}_{i=0}
                                            \binom {b_v} i
                                            \left(\frac{1}{2}\right)^{b_v}\leq\sum\limits^{t-1-a_v}_{i=0}
                                            \binom {b_v-1} i
                                            \left(\frac{1}{2}\right)^{b_v-1}$. Let $f(a,b)=\sum\limits^{t-1-a}_{i=0}\binom {b-1} i(\frac{1}{2})^{b-1}$, where $a<t$ and $2a+b=d_T^+(v)$. Now we claim that $f(a-1,b+2)>f(a,b)$, in fact,
\begin{align}
f(a-1,b+2)-f(a,b)\nonumber
&=\left(\frac{1}{2}\right)^{b+1}\left(\sum\limits^{t-a}_{i=0}\binom {b+1} i-4\sum\limits^{t-1-a}_{i=0}\binom {b-1} i\right)\\ \nonumber
&=\left(\frac{1}{2}\right)^{b+1}\left(\binom {b+1} {t-a}-2\binom {b-1} {t-a-1}-\binom {b} {t-a-1}\right). \\ \nonumber
&=\left(\frac{1}{2}\right)^{b+1}\left(\frac{(b-1)!(b-2t+2a)}{(b-t+a)!(t-a)!}\right) \\ \nonumber
&>0,\nonumber
\end{align}
where the last inequality follows from the fact $2a+b>2t$ and $a<t$. Thus we have $$Pr(X_v<t)\leq f(0,d_T^+(v))=\sum\limits^{t-1}_{i=0}\binom {d_T^+(v)-1} i\left(\frac{1}{2}\right)^{d^+_T(v)-1}
$$ and $$Pr(X_v>d_T^+(v)-t)\leq\sum\limits^{t-1}_{i=0}\binom {d_T^+(v)-1} i\left(\frac{1}{2}\right)^{d^+_T(v)-1}.$$
Suppose a random variable $Y$ has binomial distribution $\mathbb{B}(N,\frac{1}{2})$, where $N=d_T^+(v)-1$. By Chernoff's inequality, we know that $Pr(Y-E(Y)\geq N\sigma)<e^{-2N\sigma^2}$ for any positive constant $\sigma$. Thus we have \begin{align}
                                                       \sum\limits^{t-1}_{i=0}\binom{d_T^+(v)-1}i\left(\frac{1}{2}\right)^{d^+_T(v)-1}\nonumber
                                                      & =Pr(Y\geq d^+_T(v)-t) \\ \nonumber
                                                      & =Pr(Y-\frac{N}{2}\geq\frac{N}{2}-t+1) \\ \nonumber
                                                       &<e^{-2\left(d_T^+(v)-1\right)\left(\frac{1}{2}-\frac{t-1}{d_T^+(v)-1}\right)^2}\\\nonumber
                                                       &<e^{-\epsilon^2(d_T^+(v)-1)},\nonumber
                                                    \end{align}
where the last inequality follows the fact that $\frac{1}{2}-\frac{t-1}{d_T^+(v)-1}>\frac{1}{2}-\frac{t}{d_T^+(v)}>\frac{\epsilon}{\sqrt{2}}$ when $d_T^+(v)\geq\frac{2+\sqrt{2}}{\epsilon}$.
Now we bound $E(X)$.
By the linearity of expectation, $$E(X)=\sum\limits_{v\in V(T)}\{Pr(X_v>d_T^+(v)-t)+Pr(X_v<t)\}<\sum\limits_{v\in V(T)}2e^{-\epsilon^2(d_T^+(v)-1)}.$$
For every $i\in \mathbb{N}$, the number of vertices $v$ with $2^{i-1}\leq d^+_T(v)<2^i$ in $T$ is at most $2^{i+1}$, and there exists a positive integer $i_0$ such that $e^{-\epsilon^2(2^{i-1}-1)}\leq2^{-2i-2}$ whenever $i\geq i_0$. Let $\delta^+(T)\geq\delta_0:=\max\{2^{i_0-1}, \frac{2+\sqrt{2}}{\epsilon}\}$, we have
\begin{align}
  E(X)\nonumber
  &<\sum\limits_{i\geq i_0}2^{i+2}e^{-\epsilon^2(2^{i-1}-1)}\\ \nonumber
  &\leq\sum\limits_{i\geq i_0}2^{i+2}2^{-2i-2}\leq1.\nonumber
\end{align}
Thus there is a bisection of $T$ with no bad vertices, and we are done.
\end{proof}

From the proof of Theorem \ref{th1}, we derive the following corollary.

\begin{corollary}
Every tournament $T$ with $\delta^+(T)\geq(2+o(1))k$ admits a bisection $(A, B)$ with $\min\{d_A^+(v), d_{B}^+(v)\}\geq k$ for every $v\in V(T)$, where the $o(1)$-term tends to zero as $k$ tends to infinity.
\end{corollary}

\begin{proof}[{\bf Proof of Theorem \ref{thm2}}] First we introduce a well-known lemma.
\begin{lemma}\emph{\textbf{(The Weighted Local Lemma \cite{Michael})}}\label{lem1}
Consider a set $\mathcal{B}=\{A_1,A_2,\ldots,A_n\}$ of events such that each $A_i$ is mutually independent of $\mathcal{B}-(D_i\cup \{A_i\})$ for some $D_i\subset \mathcal{B}$. If we have integers $t_1,t_2,\ldots,t_n\geq1$ and a real number $0\leq p\leq\frac{1}{4}$ such that for each $i\in [n]$,
\begin{enumerate}
  \item[(a)] $Pr(A_i)\leq p^{t_i}$ and
  \item[(b)] $\sum\limits_{A_j\in D_i}(2p)^{t_j}\leq\frac{t_i}{2}$,
\end{enumerate}
then with positive probability, none of the events in $\mathcal{B}$ occur.
\end{lemma}

 The proof of Theorem \ref{thm2} relies on Lemma \ref{lem1}. We arbitrarily partition the vertices of $D$ into disjoint pairs and separate each pair independently and uniformly, then we have a bisection $(V_1, V_2)$. For a vertex $v\in V(D)$, let $x(v)$ be the number of out-neighbors of $v$ that are in the same part with $v$. Let $A(v)$ be the event that either $x(v)<s:=\lceil(\frac{1}{2}-\epsilon)d_D^+(v)\rceil$ or $x(v)>d_D^+(v)-s$, and let $\mathcal{A}=\{A(v): v\in V(D)\}$ be the set of all bad events. By the same argument as in the proof of Theorem \ref{th1}, we have $$Pr(A(v))=Pr(x(v)<s)+Pr(x(v)>d_D^+(v)-s)<2e^{-\epsilon^2(d^+_D(v)-1)}.$$
Let $t_v:=\frac{d_D^+(v)}{\delta^+(D)}$ be the associated weight. Let $p:=e^{-\epsilon^2(\delta^+(D)-1)}$ and $p<\frac{1}{4}$ whenever $\delta^+(D)$ is sufficiently large. In fact, it suffices to have $\delta^+(D)\geq \delta_0:=\min\{\delta: e^{-\epsilon^2(\delta-1)}<\frac{1}{4}, e^{\epsilon^2(\delta-1)}/8\delta\geq\delta\}$.
Now it suffices to check that conditions (a) and (b) hold. The condition (a) holds, since $$Pr(A(v))<e^{-\epsilon^2(d^+_D(v)-1)}\leq e^{-\epsilon^2(\delta^+(D)-1)d_D^+(v)/\delta^+(D)}=p^{t_v}.$$

Let $D(v)$ be the set of events that are relevant to the event $A(v)$. Therefore $A(v)$ is mutually independent of $\mathcal{A}-(D(v)\cup\{A(v)\})$. We observe that $A(v)$ and $A(u)$ are related only if $u, v$ have common out-neighbors or have neighbors in the same pair. From the observation, we have $|D(v)|\leq2d_D^+(v)\Delta^-(D)$. Since $\Delta^-(D)\leq \frac{e^{\epsilon^2(\delta^+(D)-1)}}{8\delta^+(D)}$ and $t_v\geq1$ for every $v\in V(D)$, we have $$\sum\limits_{A(w)\in D(v)}(2p)^{t_w}\leq\sum\limits_{A(w)\in D(v)}2p\leq2e^{-\epsilon^2(\delta^+(D)-1)}|D(v)|\leq\frac{t_v}{2}.$$ Now condition (b) holds, and by Lemma \ref{lem1}, with positive probability, no bad events in $\mathcal{A}$ occur. That is, we have a bisection with $s\leq x(v)\leq d_D^+(v)-s$ for every $v\in V(D)$. This observation completes the proof.\\
\end{proof}
\begin{proof}[{\bf Proofs of Theorem \ref{thm3} and Corollary \ref{cor3}}]
 Recall that a digraph with minimum out-degree $s$ is $s$-\emph{minimal} if any proper subdigraph has minimum out-degree at most $s-1$. We have the following rough characterization of minimal bipartite and $k$-partite tournaments ($k>2$).
\begin{lemma}\label{lem2}
\emph{(1)} Every $s$-minimal bipartite tournament $D$ satisfies $|V(D)|\leq \frac{(s+1)^4}{4s}$.\\
\emph{(2)} Every $s$-minimal $k$-partite tournament $D$ satisfies $|V(D)|<\max\{2s(s+1)^2, 2ks(s+1)\}$
\end{lemma}
\begin{proof}
  Let $D=(U,W)$ be an $s$-minimal bipartite tournament on $n$ vertices. It follows that for any vertex $v\in V(D)$, there is an arc $uv$ with $d_D^+(u)=s$. Define $a=|\{v\in U: d_D^+(v)=s\}|$ and $b=|\{v\in W: d_D^+(v)=s\}|$ and without loss of generality let $a\geq b$. By the fact above, we have $$s(a+b)-ab\geq n-(a+b)\geq0.$$
  So $b\leq2s$. Thus we have $n\leq\max\{ g(a,b)=(s+1)(a+b)-ab: b\leq2s, a\leq sb\}$. By monotonicity analysis, the optimal solution $(x,y)$ satisfies $x=sy$, and it follows that $n\leq\max\{ g(sy,y)=(s+1)^2y-sy^2: y\leq2s\}\leq\frac{(s+1)^4}{4s}$.

  For an $s$-minimal $k$-partite tournament $D=(U_1,U_2,\ldots,U_k)$, we denote $a_i=|\{v\in U_i: d_D^+(v)=s\}|$ for each $i\in \{1,2,\ldots,k\}$ and assume that $a_1\geq a_2\geq\ldots\geq a_k$. Similarly, we have
  \begin{equation}\label{eq1}
    s\sum\limits_{i=1}^ka_i-\sum\limits_{1\leq i<j\leq k}a_ia_j\geq n-\sum\limits_{i=1}^ka_i.\nonumber
  \end{equation}
  If $a_1\leq\sum\limits_{i=2}^ka_i$, then $2s\sum\limits_{i=2}^ka_i-a_1\sum\limits_{i=2}^ka_i>0$ by the above inequality. Thus we have $a_1<2s$ and $n<2ks(s+1)$. If $a_1>\sum\limits_{i=2}^ka_i$, then $2sa_1-a_1\sum\limits_{i=2}^ka_i>0$ and it follows that $\sum\limits_{i=2}^ka_i<2s$. By the fact that for any vertex $v\in V(D)$, there is an arc $uv$ with $d_D^+(u)=s$, we have $a_1<s\sum\limits_{i=2}^ka_i$ and $n<2s(s+1)^2$. So $n<\max\{2s(s+1)^2, 2ks(s+1)\}$.
\end{proof}
Lemma \ref{lem2} implies Theorem \ref{thm3} and Corollary \ref{cor3} directly.
\end{proof}


%


\section{Remark}
We want to mention that Alon et al. \cite{alon2} obtained a similar result regarding Theorem \ref{th1}, and their work was available on arXiv just before we submit our manuscript. The results in two papers are finished independently.






\section{Acknowledgements}
We are very grateful to the reviewers for their useful comments. This work was supported by NSFC (Nos. 11601430, 11631014, 11471193), the Foundation for Distinguished Young Scholars of Shandong Province (JQ201501), and China Postdoctoral Science Foundation (No. 2016M590969).

\end{document}